\documentclass{article}

\usepackage{amsthm}
\usepackage{amssymb}
\usepackage{amsmath}
\usepackage{url}

\newtheorem{theorem}{Theorem}[section]
\newtheorem{corollary}[theorem]{Corollary}

\theoremstyle{remark}
\newtheorem{remark}{Remark}[section]
\theoremstyle{definition}

\theoremstyle{definition}
\newtheorem{example}{Example}[section]

\begin{document}

\markboth{Regularity of solutions to higher-order integrals of the
calculus of variations}{Regularity of solutions to higher-order
integrals of the calculus of variations}

\title{Regularity of solutions to higher-order integrals\\
of the calculus of variations}

\author{Moulay Rchid Sidi Ammi\\
        \url{sidiammi@mat.ua.pt}
        \and
        Delfim F. M. Torres\\
        \url{delfim@mat.ua.pt}}

\date{Department of Mathematics\\
University of Aveiro\\
3810-193 Aveiro, Portugal}

\maketitle

\begin{abstract}
We obtain new regularity conditions for problems of calculus of
variations with higher-order derivatives. As a corollary, we get
non-occurrence of the Lavrentiev phenomenon. Our main regularity
result asserts that autonomous integral functionals with a
Lagrangian having coercive partial derivatives with respect to the
higher-order derivatives admit only minimizers with essentially
bounded derivatives.
\end{abstract}


\smallskip


\smallskip


\smallskip


\medskip


\section{Introduction and preliminaries}

Let $\mathcal{L}(t, x^{0},\ldots, x^{m})$ be a given $C^{1}([a, b]
\times \mathbb{R}^{(m+1)\times n})$ real function. The problem of
the calculus of variations with high-order derivatives consists in
minimizing an integral functional
\begin{equation}
\label{Pm} I[x(\cdot)]= \int_{a}^{b}\mathcal{L}\left(t, x(t),
\dot{x}(t), \ldots, x^{(m)}(t)\right) dt  \tag{$P_{m}$}
\end{equation}
over a certain class $\mathcal{X}$ of functions $x : [a, b]
\rightarrow \mathbb{R}^{n}$ satisfying the boundary conditions
\begin{equation}
\label{eq:boundaryCond} x(a)= x_{a}^{0}\, ,  x(b)= x_{b}^{0}\, ,
\ldots\, , x^{(m-1)}(a)= x_{a}^{m-1}\, , x^{(m-1)}(b)= x_{b}^{m-1}
\, .
\end{equation}
Often it is convenient to write $x^{(1)}= x'$, $x^{(2)}= x''$, and
sometimes we revert to the standard notation used in mechanics:
$x'=\dot{x}$, $x''= \ddot{x}$. Such problems arise, for instance,
in connection with the theory of beams and rods \cite{smi}.
Further, many problems in the calculus of variations with
higher-order derivatives describe important optimal control
problems with linear dynamics \cite{SarychevTorres2}.

Regularity theory for optimal control problems is a fertile field
of research and a source of many challenging mathematical issues
and interesting applications \cite{cla,torres2,torres3,Margarida}.
The essential points in the theory are: (i) existence of
minimizers and (ii) necessary optimality conditions to identify
those minimizers.

The first systematic approach to existence theory was introduced
by Tonelli in 1915 \cite{ton}, who showed that existence of
minimizers is guaranteed in the Sobolev space $W^{m}_{m}$ of
absolutely continuous functions. The direct method of Tonelli
proceeds in three steps: (i) regularity, and convexity with
respect to the highest-derivative of the Lagrangian $\mathcal{L}$,
guarantees lower semi-continuity; (ii) the coercivity condition
(the Lagrangian $\mathcal{L}$ must grow faster than a linear
function) implies that minimizing sequences lie in a compact set;
(iii) thus, by the compactness principle, one gets directly from
(i) and (ii) the existence of minimizers for the problem
\eqref{Pm}. Typically, Tonelli's existence theorem for \eqref{Pm}
is formulated as follows.\footnote{In our context (H1) is
trivially satisfied since we are assuming $\mathcal{L}$ to be a
$C^1$ function. It is customary to choose the function $\theta$ in
hypothesis (H3) as $\theta(r) = a r^2 + b$ for some strictly
positive constants $a$ and $b$. We then say that $\mathcal{L}$ is
\emph{quadratically coercive}.}

\begin{theorem} (see \textrm{e.g.} \cite{cla,clavin90,bookVinter})
Under hypotheses (H1)-(H3) on the
Lagrangian $\mathcal{L}$,
\begin{enumerate}
\item[(H1)] $\mathcal{L}(t, x^{0},\ldots, x^{m})$
is locally Lipschitz in $(t, x^{0},\ldots, x^{m})$;
\item[(H2)] $\mathcal{L}(t, x^{0},\ldots, x^{m})$
is convex as a function of the last argument $x^{m}$;
\item[(H3)] $\mathcal{L}(t, x^{0},\ldots, x^{m})$
is coercive in $x^{m}$, \textrm{i.e.}
$\exists$ $\theta : [0, \infty)\rightarrow \mathbb{R}$ such that
\begin{gather*}
\lim_{r \rightarrow \infty} \frac{\theta (r)}{r}=+\infty \, , \\
\mathcal{L}(t, x^{0},\ldots, x^{m}) \geq \theta (|x^{m}|) \mbox{
for all } (t, x^{0},\ldots, x^{m}) \, ,
\end{gather*}
\end{enumerate}
there exists a minimizer to problem \eqref{Pm} in the class
$W^{m}_{m}$.
\end{theorem}

The main necessary condition in optimal control is the famous
Pontryagin maximum principle, which includes all the classical
necessary optimality conditions of the calculus of variations
\cite{Pontryagin}. It turns out that the hypotheses (H1)-(H3) do
not assure the applicability of the necessary optimality
conditions, being required more regularity on the class of
admissible functions \cite{BallMizel}. For \eqref{Pm}, the
Pontryagin maximum principle \cite{Pontryagin} is established
assuming $x \in W_{m}^{\infty} \subset W^{m}_{m}$.

In the case $m=1$, extra information about the minimizers was
proved, for the first time, by Tonelli himself \cite{ton}. Tonelli
established that, under the hypotheses (H2) and (H3) of convexity
and coercivity, the minimizers $x$ have the property that
$\dot{x}$ is locally essentially bounded on an open subset $\Omega
\subset [a, b]$ of full measure. If the following Tonelli-Morrey
regularity condition \cite{clavin,SarychevTorres2,FrancisClarke}
\begin{equation}
\label{eq:conTonelliMorrey} \left|\frac{\partial
\mathcal{L}}{\partial x}\right| + \left|\frac{\partial
\mathcal{L}}{\partial \dot{x}}\right| \le c |\mathcal{L}| + r \, ,
\end{equation}
 is satisfied for some constants $c$ and $r$, $c > 0$, then $\Omega = [a,b]$
($\dot{x}(t)$ is essentially bounded in all points $t$ of $[a,b]$,
\textrm{i.e.} $x \in W_{1}^{\infty}$), and the Pontryagin maximum
principle, or the necessary condition of Euler-Lagrange, hold.
Since L.~Tonelli and C.~B.~Morrey \cite{Morrey}, several
Lipschitzian regularity conditions were obtained for the problem
\eqref{Pm} with $m = 1$: S.~Bernstein \cite{Bernstein}, for the
scalar case $n = 1$, F.~H.~Clarke and R.~B.~Vinter
\cite{clavin85}, for the vectorial case $n > 1$, obtained  the
condition
\begin{equation*}
\left|\left(\frac{\partial^2 \mathcal{L}}{\partial
\dot{x}^2}\right)^{-1} \left( \frac{\partial \mathcal{L}}{\partial
x} -\frac{\partial^2 \mathcal{L}}{\partial \dot{x} \partial t}
-\frac{\partial^2 \mathcal{L}}{\partial \dot{x} \partial x}
\dot{x}\right)\right| \le c \left(\left|\dot{x}\right|^3 + 1\right)
\, , \quad \frac{\partial^2 \mathcal{L}}{\partial \dot{x}^2} > 0 \,
;
\end{equation*}
F.~H.~Clarke and R.~B.~Vinter \cite{clavin85} the regularity
conditions
\begin{equation}
\label{eq:aut} \left|\frac{\partial \mathcal{L}}{\partial t}\right|
\le c \left|\mathcal{L}\right| + k(t) \, , \quad k(\cdot) \in L_1 \,
,
\end{equation}
and
\begin{equation*}
\left|\frac{\partial \mathcal{L}}{\partial x}\right| \le c
\left|\mathcal{L}\right| + k(t) \left|\frac{\partial
\mathcal{L}}{\partial \dot{x}}\right| + m(t) \, , \quad k(\cdot), \,
m(\cdot) \in L_1 \, ;
\end{equation*}
and A.~V.~Sarychev and D.~F.~M.~Torres \cite{SarychevTorres1} the
condition
\begin{equation}
\label{eq:MSc} \left(\left|\frac{\partial \mathcal{L}}{\partial
t}\right| + \left|\frac{\partial \mathcal{L}}{\partial x}\right|
\right) \left|\dot{x}\right|^{\mu} \le \gamma \mathcal{L}^\beta +
\eta \, , \quad \gamma > 0 \, , \beta < 2 \, , \mu \ge \max
\left\{\beta-1,-1\right\} \, .
\end{equation}
Lipschitzian regularity theory for the problem of the calculus of
variations with $m = 1$ is now a vast discipline (see
\textrm{e.g.}
\cite{CellinaFerriero,DalMaso,Cellina,Ornelas,torres3} and
references therein). Results for $m > 1$ are scarcer: we are aware
of the results in \cite{clavin90,SarychevTorres1,torres1}. In 1997
A.V.~Sarychev \cite{Sarychev} proved that the second-order
problems of the calculus of variations may show new phenomena
non-present in the first-order case: under the hypotheses
(H1)-(H3) of Tonelli's existence theory, autonomous problems
\eqref{Pm} with $m = 2$ may present the Lavrentiev phenomenon
\cite{Lav}. This is not a possibility for $m = 1$, as shown by the
Lipschitzian regularity condition \eqref{eq:aut}. Sarychev's
result was recently extended by A.~Ferriero \cite{Ferriero} for
the case $m > 2$. It is also shown in \cite{Ferriero} that, under
some standard hypotheses, the problems of the calculus of
variations \eqref{Pm} with Lagrangians only depending on two
consecutive derivatives $x^{(\gamma)}$ and $x^{(\gamma+1)}$,
$\gamma \ge 0$, do not exhibit the Lavrentiev phenomenon for any
boundary conditions \eqref{eq:boundaryCond} (for $m = 1$, this
follows immediately from \eqref{eq:aut}). In the case in which the
Lagrangian only depends on the higher-order derivative $x^{(m)}$,
it is possible to prove more \cite[Corollary~2]{SarychevTorres1}:
when $\mathcal{L} = \mathcal{L}\left(x^{(m)}\right)$, all the
minimizers predicted by the existence theory belong to the space
$W_{m}^{\infty} \subset W^{m}_{m}$ and satisfy the Pontryagin
maximum principle (regularity). As to whether this is the case or
not for Ferriero's problem with Lagrangians only depending on
consecutive derivatives $x^{(\gamma)}$ and $x^{(\gamma+1)}$, seems
to be an open question.

The results of Sarychev \cite{Sarychev} and Ferriero
\cite{Ferriero} on the Lavrentiev phenomenon show that the
problems of the calculus of variations with higher-order
derivatives are richer than the problems with $m = 1$, but also
show, in our opinion, that the regularity theory for higher-order
problems is underdeveloped. One can say that the Lipschitzian
regularity conditions found in the literature for the higher-order
problems of the calculus of variations are a generalization of the
above mentioned conditions for $m = 1$: \cite{clavin90}
generalizes \cite{clavin85} for $m > 1$, \cite{SarychevTorres1}
generalizes \eqref{eq:MSc} for problems of optimal control with
control-affine dynamics, \cite{torres1} generalizes
\eqref{eq:conTonelliMorrey} for optimal control problems with more
general nonlinear dynamics. To the best of our knowledge, there
exist no regularity conditions for the higher-order problems of
the calculus of variations of a different type from those also
obtained (also valid) for the first-order problems. We prove here
a new regularity condition which is of a different nature than
those appearing for the first-order problems. The results of the
paper extend those found in \cite{siddel}, covering problems of
the calculus of variations with derivatives of higher order than
two. While existence follows by imposing coercivity to the
Lagrangian $\mathcal{L}$ (hypothesis (H3)), we prove (\textrm{cf.}
Theorem~\ref{theorem4}) that for the autonomous high-order
problems of the calculus of variations, regularity follows by
imposing a superlinear condition with respect to the sum of the
partial derivatives $\frac{\partial \mathcal{L}}{\partial
{x}^{(m)}_i}$ of the Lagrangian. We observe that our condition is
intrinsic to the higher-order problems: for autonomous problems of
the calculus of variations with $m = 1$ \eqref{eq:aut} is
trivially satisfied and no coercivity on the partial derivatives
$\frac{\partial \mathcal{L}}{\partial \dot{x}}$ are needed. Our
condition is, however, necessary, as a consequence of Sarychev's
results \cite{Sarychev}.


\section{Outline of the paper and hypotheses}

In Section~\ref{sec:nec:cond} we establish a generalized integral
form of duBois-Reymond necessary condition, valid in the class
$\mathcal{X} = W^{m}_{m}$ (we recall that the optimal solutions
$x$ may have unbounded derivatives). In Section~\ref{sec:Reg} we
make use of our duBois-Reymond necessary condition to obtain
regularity conditions under which all the minimizers of \eqref{Pm}
are in $W_{m}^{\infty} \subset W^{m}_{m}$ and thus satisfy the
classical necessary conditions. Then, in Section~\ref{subsec:EL},
arguments analogous to those used to prove the dubois-Reymond
necessary condition in \S\ref{sec:nec:cond} are used to prove an
Euler-Lagrange necessary condition valid for non-regular
minimizers $W^{m}_{m}$ not in $W_{m}^{\infty}$. In general terms,
one can say that the techniques used here are an extension of
those appearing in \cite{ces} and \cite{siddel}.

\medskip

In the sequel we shall assume the following hypotheses, where
$x(\cdot) \in W^{m}_{m}$ is the minimizer under consideration:

\begin{description}

\item[$(S_{0})$] There exists a continuous function $S\geq 0$,
$(t, x^{0}, x^{1},\ldots, x^{m} ) \in \mathbb{R}^{1+ (m+1)\times
n}$, and some $\delta
> 0$, such that $t \rightarrow S(t, x^{0}(t), \ldots, x^{m}(t))$ is
$L^{m'}$-integrable in $[a, b]$, $m'$ being the Holder conjugate
of $m$ ($\frac{1}{m}+\frac{1}{m'}=1$) and
$$
\left|\frac{\partial \mathcal{L}}{\partial t}(\tau, x^{0}, x^{1},
\ldots, x^{m})\right| \leq S(t, x^{0}, x^{1}, \ldots, x^{m}),
$$
for any $t \in [a, b]$, $\left|\tau -t\right|< \delta$,
$x^{i}=x^{i}(t)$, $0\leq i \leq m$.

\item[$(S_{i})$] There exists a nonnegative continuous function $G$,
and some $\delta >0$, such that $t \rightarrow G(t, x^{0}(t),
x^{1}(t),\ldots,  x^{m}(t))$ is $L^{m'}$-integrable on $[a, b]$, and
\begin{equation*}
\begin{split}
\left|\frac{\partial \mathcal{L}}{\partial x_{i}^{(k)}} (t, x^{0},
x^{1},\ldots, x^{k-1}, y, x^{k+1}, \ldots,  x^{m})\right| &\leq G(t,
x^{0}, \ldots, x^{m})\, ,
 \end{split}
\end{equation*}
for any $t \in [a, b]$, $x$, $x^{1}$, $\ldots$,  $x^{m} \in
\mathbb{R}^{n}$, $x=(x_{1},\ldots, x_{n})\in \mathbb{R}^{n}$,
$y=(y_{1},\ldots, y_{n})\in \mathbb{R}^{n}$, $y_{j}= x^{k}_{j}(t)$
for $j\neq i$,  $ \left|y_{i}-x_{i}^{k}(t)\right|\leq \delta$,
$i=1, \ldots, n$ and $k=0, 1, 2, \ldots, m$, where $x_{i}^{k}(t)$
is the $i^{th}$ component of the $k^{th}$ vector with the
convention $x_{i}^{0}(t)=x_{i}(t)$.
\end{description}

\begin{remark}
Hypothesis $(S_{0})$ is certainly verified if $\mathcal{L}$ does
not depend on $t$: $(S_{0})$ holds trivially in the autonomous
case. Conditions $(S_{i})$, $i = 0,\ldots,n$, are needed in the
proof of Theorems~\ref{theorem1} and \ref{theorem2} to justify the
usual rule of differentiation under the sign of an integral.
\end{remark}


\section{Generalized duBois-Reymond equation}
\label{sec:nec:cond}

In this section we prove an integral form of the duBois Reymond
equation (equality \eqref{eq2} of Theorem~\ref{theorem1} below).
For this, we consider an arbitrary change of the independent
variable $t$. Let $s$ be the arc length parameter on the curve
$C_{0}: x = x(t)$, $a \leq t \leq b$, so that the Jordan length of
$C_{0}$ is $s(t)=\int_{a}^{t} \sqrt{ 1+(x'(\tau))^{2}}d\tau$ with
$s(a)=0$, $s(b)=l$ and $s(t)$ is absolutely continuous with
$s'(t)\geq 1$ $a.e.$ Thus $s(t)$ and its inverse $t(s)$, $0\leq s
\leq l$, are absolutely continuous with $t'(s)>0$ $a.e.$ in $[0,
l]$. If $X(s)= x(t(s))$, $0 \leq s \leq l$, then $t(s)$ and $X(s)$
are Lipschitzian of constant one in $[0, l]$. By the usual change
of variable,
\begin{equation*}
\begin{split}
I[x] &= \int_{a}^{b}\mathcal{L}\left(t, x(t),\dot{x}(t) , \ldots, x^{(m)}(t) \right) dt \\
&= \int_{0}^{l} \mathcal{L}\Biggl(t(s), X(s),\frac{X'(s)}{t'(s)},
\sum_{i=1}^{2}P_{i2}(t', t'')X^{(i)}, \\
&\qquad\qquad\qquad \dots,\sum_{i=1}^{m}P_{im}(t', t'', \ldots, t^{(m)})X^{(i)}\Biggr)t'(s) ds
\, ,
\end{split}
\end{equation*}
where $P_{ik}$, $1\leq k \leq m$, are  functions on $(t', t'',
\ldots, t^{(k)})$, obtained by differentiating $X(s)$ $k-$times
and replacing the derivatives $x^{i}(t(s))$ by $X^{i}(s)$,
$i=1,\dots, k-1$. Setting
\begin{multline*}
F(t, x, t', x', t'', x'', \ldots, t^{(m)}, x^{(m)}) \\
= \mathcal{L}\Biggl(t, x,\frac{x'}{t'}, \sum_{i=1}^{2}P_{i2}(t',
t'')x^{(i)}, \ldots, \sum_{i=1}^{k}P_{ik}(t', t'', \ldots,t^{(k)})x^{(i)},\\
\ldots \sum_{i=1}^{m}P_{im}(t', t'', \ldots,
t^{(m)})x^{(i)} \Biggr)t',
\end{multline*}
then we have:
$$
I[x]= J[C]= J[X] =\int_{0}^{l}F\left(t(s), X(s), t'(s), X'(s),
\ldots,  t^{(m)}(s), X^{(m)}(s)\right) ds \, .
$$
Thus, after reparameterization by length, the cost functional can
be considered as a functional $J[C]$ in the space of curves,
rather than a functional in the space of functions $W^{m}_{m}$.

\begin{remark}
For $m=2$, we have $F(t, x, t', x', t'', x'') =
\mathcal{L}\left(t, x,\frac{x'}{t'},\frac{1}{t^{'2}}x''-
\frac{t''}{t'^{3}}x'\right)t'$.
\end{remark}

The following necessary condition will be useful to prove our
regularity theorem (Theorem~\ref{theorem4}).

\begin{theorem}\label{theorem1}
Under hypotheses $(S_{i})_{0 \leq i \leq n}$, if $x(\cdot) \in
W_m^m$ is a minimizer of problem \eqref{Pm}, then the following
integral form of duBois-Reymond necessary condition holds:
\begin{align}
&\phi_{0}(s)= \frac{\partial F}{\partial t^{(m)}}
 + \sum_{i=1}^{m} (-1)^{m-i+1}\int_{0}^{s} \int_{0}^{\tau_{1}} \ldots
  \int_{0}^{\tau_{m-i}} \frac{\partial F}{\partial t^{(i-1)}}d\sigma \,
    d\tau_{m-i} \ldots d\tau_{1} =c_{0}
 , \label{eq2}
\end{align}
where $0 \leq \tau_{i} \leq s \leq l$, $c_{0}$ is a constant, and
functions $\frac{\partial F}{\partial t^{(i)}}$, $1\leq i \leq m$,
are evaluated at $(t(s), X(s), t'(s), X'(s), \ldots t^{(m)}(s),
X^{(m)}(s))$.
\end{theorem}

\begin{remark}
For $m=2$ \eqref{eq2} takes the following form:
\begin{equation*} \phi_{0}(s)=
\frac{\partial F}{\partial t''}- \int_{0}^{s}\frac{\partial
F}{\partial t'}+ \int_{0}^{s}\int_{0}^{\tau}\frac{\partial
F}{\partial t}= c_{0}, \quad 0\leq \tau \leq s \leq l \, .
\end{equation*}
\end{remark}

\begin{proof}
It is to be noted that $(t(s), X(s), t'(s), \ldots,  X^{(m)}(s),
t^{(m)}(s))$ may not exist in a set of null-measure of all $s$. The
proof is done by contradiction. Suppose that \eqref{eq2} is not
true. Then, there exist constants $d_{1}<d_{2}$ and disjoints sets
$E_{1}^{*}$ and $E_{2}^{*}$ of non-zero measure such that
\begin{gather}
\phi_{0}(s)\leq d_{1} \mbox{ for } s \in
E_{1}^{*}, \nonumber\\
\phi_{0}(s)\geq  d_{2} \mbox{ for } s \in E_{2}^{*}, \nonumber
\end{gather}
while $t'(s)>0$ $a.e$ in $[0, l]$. Hence, there exist some
constant $k>0$ and two subsets $E_{1}$, $E_{2}$ of positive
measure of $E_{1}^{*}$, $E_{2}^{*}$, such that
\begin{gather}
t'(s)\geq k>0, \quad \phi_{0}(s)\leq d_{1} \quad \mbox{ for } s \in
E_{1}, \quad \left|E_{1}\right|>0 \, , \label{eq3}\\
t'(s)\geq k>0,  \quad \phi_{0}(s)\geq  d_{2}  \quad \mbox{ for } s
\in E_{2},  \quad \left|E_{2}\right|>0 \, . \label{eq4}
\end{gather}
Let us consider
$$
\psi(s)= \int_{0}^{s}\int_{0}^{\tau_{1}} \ldots
\int_{0}^{\tau_{m-1}}\{ \left|E_{2}\right| \chi_{1}
-\left|E_{1}\right| \chi_{2}\}\,  d\sigma d{\tau_{m-1}} \ldots
d{\tau_{1}},
$$
 $0\leq \tau_{i} \leq  s \leq l, \, 1\leq i\leq
m-1,$ where $\chi_{j}$ denotes the indicator function defined by
 $$
 \chi_{j}(s)=
 \left\{
 \begin{array}{rll}
  1  & \mbox{ for } & s \in E_{j},
\\
  0  & \mbox{ for } & s \in [0, l]/E_{j}, \quad  j=1, 2 \mbox{ and }  0\leq s \leq
  l.
   \end{array}
    \right.
   $$
We have that $\psi^{(m-1)}$ is an absolutely continuous function in
$[0, l]$ with $\psi^{(m-1)}(0)= \psi^{(m-1)}(l)= 0$. Moreover,
$$
 \psi^{(m)}(s)=
 \left\{
 \begin{array}{rll}
  -\left|E_{1}\right| & \mbox{ $a.e$ } & s \in E_{2} \, ,
\\
 \left|E_{2}\right| & \mbox{ $a.e$ } & s \in E_{1} \, , \\
0 & \mbox{ $a.e$ } & s \in [0, l]-E_{1} \bigcup E_{2} \, .
   \end{array}
    \right.
   $$
We also define $C_{\alpha}: t= t_{\alpha}(s)$, $x= X_{\alpha}(s)$,
$0 \leq s  \leq l$, by setting
$$
t_{\alpha}(s)= t(s)+ \sum_{i=1}^{m}\alpha^{i} \psi^{(i-1)} (s) \, ,
$$
$$
X_{\alpha}(s)= X(s), \quad 0 \leq s \leq l, \quad
\left|\alpha\right|\leq 1 \, .
$$
Let $\rho >0$ be chosen in such a way that  $t, \tau \in [a, b]$
and $\left|t-\tau\right|< \rho$ imply $|x(t)-x(\tau)|\leq \delta$,
where $\delta$ is the constant in condition $(S_{0})$. We now
choose $\alpha$ small enough, $|\alpha| \leq \alpha_{0}$, to give
$t_{\alpha}'(s)> 0$ for $s \in E_{1}\bigcup E_{2}$, and
$C_{\alpha}$ has an absolutely continuous representation
$x=x_{\alpha}(t)$, $a \leq t \leq b$. We also have
$|t_{\alpha}(s)-t(s)|< \rho$. Hence $|x_{\alpha}(t)-x(t)|=
|x(t_{\alpha}(s))-x(t(s))|< \delta$ and we conclude that
$J[C_{\alpha}]\geq J[C]$.  On the other hand, by setting
$\phi(\alpha, s) = F(t, X, t', X', \ldots t^{(m)}, X^{(m)})$, we
have by differentiation that
$$ \left.\frac{\partial \phi}{\partial
\alpha}\right|_{\alpha = 0}= \frac{\partial F}{\partial t} \psi +
\frac{\partial F}{\partial t'} \psi'+  \ldots + \frac{\partial
F}{\partial t^{(m)}} \psi^{(m)} \, ,
$$
where
\begin{equation}
\label{eq5}
\begin{gathered}
\frac{\partial F}{\partial t} = \frac{\partial \mathcal{L}}{\partial t}t'  \, ,  \\
\vdots \\
 \frac{\partial F}{\partial t^{(k)}}= t'\frac{\partial
\mathcal{L}}{\partial x^{(k)}} \sum_{i=1}^{k} \frac{\partial
P_{ik}}{\partial t^{(k)}}x^{(i)}+ \ldots + t'\frac{\partial
\mathcal{L}}{\partial x^{(m)}} \sum_{i=1}^{k} \frac{\partial
P_{im}}{\partial t^{(k)}}x^{(i)} \\
\vdots \\
\frac{\partial F}{\partial t^{(m-1)}}= t'\frac{\partial
\mathcal{L}}{\partial x^{(m-1)}} \sum_{i=1}^{m-1} \frac{\partial
P_{i(m-1)}}{\partial t^{(m-1)}}x^{(i)} + t'\frac{\partial
\mathcal{L}}{\partial x^{(m)}} \sum_{i=1}^{m}
\frac{\partial P_{im}}{\partial t^{(m-1)}}x^{(i)}  \\
 \frac{\partial F}{\partial t^{(m)}} = t'\frac{\partial
\mathcal{L}}{\partial x^{(m)}}  \sum_{i=1}^{m} \frac{\partial
P_{im}}{\partial t^{(m)}}x^{(i)}  \, .
\end{gathered}
\end{equation}
By hypotheses $(S_{i})_{0\leq i \leq n}$, both absolutes value of
terms $\frac{\partial F}{\partial t}\psi, \frac{\partial
F}{\partial t'}\psi', \ldots, \frac{\partial F}{\partial
t^{(m)}}\psi^{(m)} $ are bounded  in $E_{1} \bigcup E_{2}$  by a
fixed function which is $L-$integrable in $[0, l]$. Then, we can
differentiate under the sign of the integral to obtain:
$$
0= \left.\frac{\partial J(C_{\alpha})}{d \alpha}\right|_{\alpha= 0}
= \int_{0}^{l}\left( \frac{\partial F}{\partial t} \psi +
\frac{\partial F}{\partial t'} \psi'+ \ldots, \frac{\partial
F}{\partial t^{(m)}} \psi^{(m)}\right)ds \, .
$$
Integration by parts, and using \eqref{eq3}--\eqref{eq4}, yields
\begin{equation*}
\begin{split}
0 &= \int_{0}^{l} \phi_{0}(s)\psi^{(m)} ds =
\int_{E_{1}}\phi_{0}(s)\psi^{(m)} ds
+ \int_{E_{2}}\phi_{0}(s)\psi^{(m)} ds \\
& \leq |E_{2}| |E_{2}|(d_{1}-d_{2}) < 0
\end{split}
\end{equation*}
which is a contradiction. Equality \eqref{eq2} is proved.
\end{proof}


\section{Main results}

In \S\ref{sec:Reg} we obtain a new regularity result which implies
the validity of the classical Euler-Lagrange necessary condition.
In \S\ref{subsec:EL} a new Euler-Lagrange necessary condition is
proved which is valid both for regular and non-regular minimizers.


\subsection{Regularity for autonomous problems}
\label{sec:Reg}

We shall present now a regularity result for \eqref{Pm} under
certain additional requirements on the Lagrangian $\mathcal{L}$.

\begin{theorem}
\label{theorem4} In addition to the hypotheses $(S_{i})_{0\leq i
\leq n}$, let us consider the autonomous problem \eqref{Pm},
\textrm{ i.e.} let us assume that $\mathcal{L}$ does not depend on
$t$: $\mathcal{L} = \mathcal{L}(x,\dot{x}, \ldots, {x}^{(m)})$. If
$\frac{\partial \mathcal{L}}{\partial {x}^{(m)}}$ is superlinear
with respect to the sum of the derivatives $x^{(i)}, i=1, \ldots,
m$, \textrm{i.e.} there exist constants $a
>0$ and $b
> 0$ such that
\begin{equation}
\label{H4} a \sum_{i=1}^{m}|x^{(i)}|+ b \leq \left|\frac{\partial
\mathcal{L}}{\partial x^{(m)}}(x, \dot{x}, \ldots,
x^{(m)})\right| \, ,
\end{equation}
then every minimizer $x \in W^{m}_{m}$ of the problem is on
$W_{m}^{\infty}$.
\end{theorem}

\begin{corollary}
Under the hypotheses of Theorem~\ref{theorem4}, the autonomous
problem of the calculus of variations with higher-order
derivatives do not admit the Lavrentiev gap
$W^{m}_{m}-W_{m}^{\infty}$:
\begin{multline*}
\inf_{x(\cdot) \in W^{m}_{m}} \int_{a}^{b}
\mathcal{L}\left(x(t),\dot{x}(t) , \ldots,  {x}^{(m)}(t)\right) dt \\
= \inf_{x(\cdot) \in W_{m}^{\infty}} \int_{a}^{b}
\mathcal{L}\left(x(t),\dot{x}(t) , \ldots, {x}^{(m)}(t)\right) dt \, .
\end{multline*}
\end{corollary}

\begin{proof}
Using  \eqref{eq2}, \eqref{eq5}, the fact that we consider the
autonomous case, and applying Holder's inequality, we get
\begin{equation*}
\begin{gathered}
\left|t'\frac{\partial \mathcal{L}}{{\partial
{x}^{(m)}}}\sum_{i=1}^{m}\frac{\partial {P}_{im}}{{\partial
{t}^{(m)}}}x^{(i)} \right|  \leq c_{0}+ c_{1}
 + \left|t' \int_{0}^{s}\sum_{i=1}^{m}\frac{\partial P_{im} }{\partial t^{(m-1)}}x^{(i)}
 \frac{\partial \mathcal{L}}{\partial x^{(m)}} \right| \\
 +  \left|  \int_{0}^{s}t' \sum_{i=1}^{m-1}\frac{\partial P_{i(m-1)} }{\partial t^{(m-1)}}x^{(i)}
 \frac{\partial \mathcal{L}}{\partial x^{(m-1)}} \right|,
 \end{gathered}
\end{equation*}
for positive constants $c_{0}$ and $c_{1}$. Therefore, with the aid
of the condition $(S_{i})$, we have
\begin{equation*}
\begin{gathered}
\left|\frac{\partial \mathcal{L}}{{\partial
{x}^{(m)}}}\sum_{i=1}^{m}x^{(i)} \right|  \leq c_{3}
 +  c_{4}\int_{0}^{s}\left| \sum_{i=1}^{m}{x}^{(i)}
 \frac{\partial \mathcal{L}}{\partial x^{(m)}} \right|
 \end{gathered}
\end{equation*}
where $c_{3}$ and $c_{4}$ are positive constants. Then, using the
fact that $\mathcal{L}\in C^{1}$, $\mathcal{L}$, $\frac{\partial
\mathcal{L}}{\partial \dot{x}}$, $\ldots$,  $\frac{\partial
\mathcal{L}}{\partial {x}^{(m)}}$ $\in L^{m'}$ and $x \in W_{m}^{m}$
(in other terms, $x, \dot{x}, \ldots, {x}^{(m)}\in L^{m}$), it
follows by the Gronwall lemma that $ \frac{\partial
\mathcal{L}}{\partial {x}^{(m)}}\sum_{i=1}^{m}{x}^{(i)}$ satisfies a
condition of the form
$$
\left|\frac{\partial \mathcal{L}}{{\partial
{x}^{(m)}}}\sum_{i=1}^{m}x^{(i)} \right|  \leq c_{5}
$$
for a certain  positive constant $c_{5}$. Besides, since
$\frac{\partial \mathcal{L}}{\partial {x}^{(m)}}$ verifies
\eqref{H4}, we have
$$
\left(a\sum_{i=1}^{m}|x^{(i)}|+ b\right)\left|\sum_{i=1}^{m}x^{(i)}\right|
\leq c_{5} \quad (b>0) \, .
$$
Therefore, $\sum_{i=1}^{m}|x^{(i)}|$  and
$\frac{\partial \mathcal{L}}{\partial {x}}$ are uniformly bounded.
This implies that $|{x}^{(i)}|$, $1\leq i \leq m$, are essentially bounded.
\end{proof}


\subsection{Generalized Euler-Lagrange equation}
\label{subsec:EL}

If Theorem~\ref{theorem4} holds, then one can use the classical
Euler-Lagrange equation to obtain the minimizers. If this is not
the case, \textrm{i.e.} Theorem~\ref{theorem4} does not hold, it
may happen that minimizers fail to satisfy the standard
Euler-Lagrange equations. Next we give a generalized
Euler-Lagrange necessary condition which is valid in the class of
functions $W^{m}_{m}$ where existence is proved.

\begin{theorem}
\label{theorem2} Under the hypotheses $(S_{i})_{0\leq i \leq n}$, if
$x(\cdot) \in W_m^m$ is a minimizer of problem \eqref{Pm}, then we
have the following integral form of the Euler-Lagrange equations:
\begin{align}
&\phi_{i}(s)= \frac{\partial F}{\partial x_{i}^{(m)}}
 + \sum_{j=1}^{m} (-1)^{m-j+1}\int_{0}^{s} \int_{0}^{\tau_{1}} \ldots
  \int_{0}^{\tau_{m-j}} \frac{\partial F}{\partial x_{i}^{(j-1)}}d\sigma \,
    d\tau_{m-j} \ldots d\tau_{1} =c_{i} , \label{eq6}
\end{align}
where functions $\frac{\partial F}{\partial {x}_{i}^{(j)}}$
are evaluated at $\left(t(s), X(s),
t'(s), X'(s), \ldots , t^{(m)}(s),\right.$
$\left.X^{(m)}(s)\right)$, $c_{i}$
denote constants, $i = 1,\ldots,n$, and
\begin{equation*}
\begin{split}
\frac{\partial F}{\partial {x}_{i}} &=
 t' \frac{\partial \mathcal{L} }{\partial {x}_{i}} \\
 \frac{\partial
F}{\partial \dot{x}_{i}} &= \frac{\partial \mathcal{L}}{\partial
\dot{x}_{i}} +t'\left (P_{12}\frac{\partial \mathcal{L}}{\partial
\ddot{x}_{i}}+ \ldots+  P_{1m}\frac{\partial
\mathcal{L}}{\partial {x}_{i}^{(m)}}\right )\, ,\\
& \vdots \\
 \frac{\partial F}{\partial {x}^{(k)}_{i}} &=t' \left (
P_{kk}\frac{\partial \mathcal{L}}{\partial
x_{i}^{(k)}}+P_{k(k+1)}\frac{\partial \mathcal{L}}{\partial
x_{i}^{(k+1)}}+ \ldots + P_{km}\frac{\partial \mathcal{L}}{\partial
x_{i}^{(m)}}\right ) \, , \\
& \vdots \\
 \frac{\partial F}{\partial {x}^{(m)}_{i}} &=
t'P_{mm}\frac{\partial \mathcal{L}}{\partial x_{i}^{(m)}} \, .
\end{split}
\end{equation*}
\end{theorem}

\begin{example}
\label{ex:CV90} Let us consider the autonomous problem proposed in
\cite{cla,clavin90} ($n=1$, $m = 2$): $\mathcal{L}(s, v, w)=
\left|s^{2}-v^{5}\right|^{2}|w|^{22}+\varepsilon |w|^{2}$, $t \in
[0,1]$. The problem satisfies hypotheses (H1)-(H3) of Tonelli's
existence theorem. Function $\tilde{x}(t)= k t^{\frac{5}{3}}$
verifies the integral form of the Euler-Lagrange equations
\eqref{eq6}. However, $\tilde{x}$ belongs to $W^{2}_{2}$ but not
to $W_{2}^{\infty}$. The regularity condition \eqref{H4} of
Theorem~\ref{theorem4} is not satisfied.
\end{example}

\begin{proof}
The proof is by contradiction and is analogous to that of
Theorem~\ref{theorem1}. Suppose that \eqref{eq6} is not satisfied.
For $i=1, \ldots, n$ and $|\alpha| \leq 1$, we consider the curve
$C_{\alpha}: t= t_{\alpha}(s)$, $x= X_{\alpha}(s)$, $0 \leq s \leq
l$, with
\begin{gather*}
X_{i \alpha}(s) = X_{i}(s) + \sum_{i=1}^{m}\alpha^{i} \psi^{(i-1)}(s) \, ,\\
X_{j \alpha}(s)= X_{j}(s), \quad j \neq i \, ,
\end{gather*}
$\psi$ defined as in the proof of Theorem~\ref{theorem1}.
We have $|\psi^{(m)}(s)| \leq l$ $a.e$ and we choose $\alpha$ small
enough in order to verify
$$
\left|X^{(k)}_{i \alpha}(s)-X^{(k)}_{i}(s)\right|  \leq \delta \, .
$$
Thus, $J[C] \leq J[C_{\alpha}]$ for all  $|\alpha|\leq
\alpha_{0}$. Setting
$$\phi(\alpha, s)=  F(t(s), X(s), t'(s), X'(s), \ldots,
 t^{(m)}(s), X^{(m)}(s))
$$
we have
$$ \left.\frac{\partial \phi}{\partial
\alpha}\right|_{\alpha = 0}= \frac{\partial F}{\partial x_{i}} \psi
 + \frac{\partial F}{\partial \dot{x}_{i}} \psi' + \ldots +
\frac{\partial F}{\partial x^{(m)}_{i}} \psi^{(m)}, \mbox{ for }
s\in [0, l] \quad a.e.
$$
By the hypotheses $(S_{i})_{0 \leq i \leq n}$ we have for
$\alpha \leq \alpha_{0}$ that $\frac{\partial \phi}{\partial \alpha}$ is,
in absolute value, bounded in $E_{1}\bigcup E_{2}$ by a
$L$-integrable function in $[0, l]$.  The proof  continues in the same lines
as in the end of the proof of Theorem~\ref{theorem1}, applying the
usual rule of differentiation under the integral sign and
integration by parts, which leads to a contradiction.
\end{proof}


\section{Conclusions}

The search for appropriate conditions on the data of the problems
of the calculus of variations with higher-order derivatives, under
which we have regularity of solutions or under which more general
necessary conditions hold, is an important area of study. In this
paper we generalize our previous results \cite{siddel} to problems
of the calculus of variations of higher order than two. We have
proved duBois-Reymond and Euler-Lagrange type necessary optimality
conditions valid in the class of functions where the existence is
proved. Minimizers in this class may have unbounded derivatives
and fail to satisfy the classical necessary conditions of
duBois-Reymond or Euler-Lagrange. We prove that if the derivatives
of the Lagrangian function with respect to the highest derivatives
are superlinear or coercive, then all the minimizers have
essentially bounded derivatives. This imply non-occurrence of the
Lavrentiev phenomenon and validity of classical necessary
optimality conditions.


\section*{Acknowledgements}

This work was supported by the \emph{Portuguese Foundation for
Science and Technology} (FCT), co-financed by the European
Community Fund FEDER/POCTI, through the \emph{Control Theory
Group} (cotg) of the \emph{Centre for Research in Optimization and
Control} (CEOC) of the University of Aveiro, and the project
SFRH/BPD/20934/2004.



\end{document}